\documentclass[12pt]{amsart}
\pdfoutput=1
\usepackage{amssymb}
\usepackage{graphicx}
\usepackage{hyperref}

\newtheorem{theorem}{Theorem}
\newtheorem{lemma}[theorem]{Lemma}
\newtheorem{proposition}[theorem]{Proposition}

\newtheorem{corollary}[theorem]{Corollary}
\newtheorem*{theorema}{Theorem}

\theoremstyle{definition}
\newtheorem{definition}[theorem]{Definition}
\newtheorem{notation}[theorem]{Notation}

\theoremstyle{remark}
\newtheorem{remark}[theorem]{Remark}

\newcommand{\ab}{\allowbreak}

\newcommand{\bC}{\mathbb{C}}
\newcommand{\E}{\textrm{E}}
\newcommand{\F}{\mathbb{F}}

\newcommand{\Tr}{\mathop{{\rm Tr}}}

\newcommand{\wt}{\widetilde}
\newcommand{\wh}{\widehat}

\title{The Non-Commutative Cycle Lemma}
\author[Armstrong]{Craig Armstrong$^{(*)}$}
\address{Queen's University, Department of Mathematics and
  Statistics, Jeffery Hall, Kingston, ON K7L 3N6, Canada}
  \thanks{$^{*}$Research supported by a USRA from the Natural
Sciences and Engineering Research Council of Canada}
\email{5ca5@queensu.ca}

\author[Mingo]{James A. Mingo$^{(\dagger)}$}
\address{Queen's University, Department of Mathematics and
  Statistics, Jeffery Hall, Kingston, ON K7L 3N6, Canada}
\thanks{$^{\dagger}$
Research supported by Discovery Grants from the Natural
Sciences and Engineering Research Council of Canada}

\email{mingo@mast.queensu.ca}

\author[Speicher]{Roland Speicher$^{(\dagger)(\ddagger)}$}
\address{Queen's University, Department of Mathematics and
  Statistics, Jeffery Hall, Kingston, ON K7L 3N6, Canada}
\email{speicher@mast.queensu.ca}
\thanks{$^{\ddagger}$
Research supported by a Killam Fellowship from
the Canada Council for the Arts.}

\author[Wilson]{Jennifer C. H. Wilson$^{(*)}$}

\address{Queen's University, Department of Mathematics and
  Statistics, Jeffery Hall, Kingston, ON K7L 3N6, Canada}

\email{4jchw@qlink.queensu.ca}

\date{}

\begin{document}


\begin{abstract}
We present a non-commutative version of the Cycle Lemma of Dvoretsky and Motzkin that
applies to free groups and use this result to solve a number of problems involving cyclic
reduction in the free group. We also describe an application to random matrices, in
particular the fluctuations of Kesten's Law.
\end{abstract}

\maketitle


\section{Introduction}
Suppose $n$ is a positive integer and $\epsilon_1, \dots, \epsilon_n$ is a string of
$+1$'s and $-1$'s. The string is said to be \textit{dominating} if for each $1 \leq i
\leq n$ the number of $+1$'s in the initial substring $\epsilon_1, \dots , \epsilon_i$ is
more than the number of $-1$'s in $\epsilon_1, \dots , \epsilon_i$.

Let $k:=\sum_{i=1}^n \epsilon_i$ be the difference between the number of $+1$'s and
$-1$'s in $\epsilon_1, \dots, \epsilon_n$.  The Cycle Lemma asserts that if $k> 0$ then
there are exactly $k$ cyclic permutations of the string $\epsilon_1, \dots, \epsilon_n$
which are dominating. The statement dates at least to J. Bertrand \cite{b} in 1887.
Dvoretsky and Motzkin \cite[Theorem 1]{dm} gave a simple and elegant proof; see also
Dershowitz and Zaks \cite{dz} for a survey of recent references and
applications.\footnote{The Cycle Lemma covers also generalizations to $p$-dominating
strings; however the case $p
  > 1$ follows by replacing each $\epsilon = -1$ with $p$
  $\epsilon$'s equal to $-1$.}

In this paper we prove a version of the Cycle Lemma for free groups, which when the group
has only one generator reduces to the result of Dvoretsky and Motzkin. In the free group
case, the non-commutative nature of the problem means that simply counting the excess of
$+1$'s to $-1$'s cannot describe the resulting configurations. Instead, we are required
to use planar diagrams.

The application to random matrices involves the concept of asymptotic freeness introduced
over twenty years ago by D. Voiculescu \cite{vdn}, who showed that the asymptotics of
certain random matrix ensembles can be described using the algebra of the free group. At
the end of the paper we shall give an indication of the problem on random matrix theory
that led us to the Non-Commutative Cycle Lemma.

Let $\F_N$ be the free group on the $N$ generators $u_1, \dots, u_N$ and let $w = l_1
\cdots l_n$ be a word in $u_1^{\pm 1}, \dots, u_N^{\pm 1}$. By a word we mean a string a
letters which may or may not simplify. The \textit{length} of a word is the number of
letters in the string. Following usual terminology, we shall say that $w = l_1 \cdots
l_k$ is reduced, or to be more precise \textit{linearly reduced}, if for all $1 \leq i <
k$, $l_i \not = l_{i+1}^{-1}$. We shall say that $w$ is \textit{cyclically reduced} if in
addition $l_k \not = l_1^{-1}$. Equivalently, $w$ is cyclically reduced if $w \cdot w$ is
linearly reduced. We say that a word $w$ \textit{reduces linearly} to a word $w'$ if $w'$
is linearly reduced and can be obtained from $w$ by successively removing neighboring
letters which are inverses of each other. We say that $w$ \textit{reduces cyclically} to
$w'$ if $w'$ is cyclically reduced and we can obtain $w'$ from $w$ by successive removal
of cyclic neighbors which are inverses of each other (i.e., in that case we might also
remove the first and the last letter if they are inverses of each other). We say that a
word $w$ is \textit{reducible to 1} if it reduces linearly to the identity in $\F_N$. One
should note that for reductions to 1 there is no difference between linear and cyclic
reducibility; any word that can be reduced cyclically to 1 can also be reduced linearly
to 1. If the reduced word $w'$ is not the identity then the situation will be quite
different for the two cases $N=1$ and $N>1$. For $N=1$ any cyclic reduction can also be
achieved in a linear way, but this is not the case for $N>1$ any more. In particular,
whereas the linear reduction of a word is always unique, this is not true any more for
the cyclic reduction. For example, the word $u_1^{-1}u_2u_2^{-1}u_1^{-1}u_2^{-1}u_1$,
which reduces linearly to $u_1^{-1}u_1^{-1}u_2^{-1}u_1$, has two different cyclic
reductions, namely $u_1^{-1}u_2^{-1}$ and $u_2^{-1}u_1^{-1}$. Since one can think of the
cyclic reduction as acting on the letters arranged on a circle, it is clear that any two
cyclic reductions are related by a cyclic permutation. Thus the length of a cyclic
reduction is well defined.

Recall that the string $\epsilon_1, \epsilon_2, \epsilon_3, \dots, \epsilon_n$ is
dominating if
\[\epsilon_1\]
\[\epsilon_1 + \epsilon_2\]
\[\epsilon_1 + \epsilon_2+ \epsilon_3\]
\[ \vdots\]
\[\epsilon_1 + \epsilon_2+ \epsilon_3 + \dots + \epsilon_n\]
are all strictly positive.

Let us translate this property to the word $u^{\epsilon_1} u^{\epsilon_2} u^{\epsilon_3}
\cdots u^{\epsilon_n}$, with $u=u_1$. Starting with any word $l_1 \cdots l_n$ with $l_i
\in \{u_1^{\pm 1}, \dots, u_N^{\pm1} \}$ we let $s_j$ be its $j$-th prefix, $s_j = l_1
\cdots l_j$. Then the dominating property of $\epsilon_1, \epsilon_2, \epsilon_3, \dots,
\epsilon_n$ is equivalent to the fact that no prefix $s_j$ ($1\leq j\leq n$) of
$u^{\epsilon_1}_1 u^{\epsilon_2}_1 u^{\epsilon_3}_1 \cdots u^{\epsilon_n}_1$ is reducible
to 1.

For example let $\epsilon_1 = \epsilon_2 = \epsilon_3 =1$ and $\epsilon_4 = \epsilon_5 =
-1$; then the only cyclic permutation of $\epsilon_1, \epsilon_2, \epsilon_3, \epsilon_4,
\epsilon_5$ that is dominating is the identity permutation. Equivalently, of the five
cyclic permutations of $u\, u\, u\, u^{-1} u^{-1}$, only $u\, u\, u\, u^{-1} u^{-1}$
(i.e. apply the identity permutation) has no prefixes which are reducible to 1.

Now let us consider the same problem when the word contains more than one generator. We
start with a word $w = l_1 \cdots l_n$ with $l_i \in \{u_1^{\pm 1}, \dots, u_N^{\pm1}
\}$. We say that $w$ has \textit{good reduction} if
\begin{enumerate}
\item no prefix of $w$ is reducible to 1 and
\item the linear reduction of $w$ is cyclically reduced. (In
  the case of only one generator, this condition is always
  satisfied.)
\end{enumerate}
We can then ask how many cyclic permutations of $w$ have good reduction; and our main
result is that this is the same as the number of letters in a cyclic reduction of $w$.

\begin{theorema}[The Non-Commutative Cycle Lemma]
Let $w$ be a word and $k$ the length of a cyclic reduction of $w$.
Then $w$ has exactly $k$ cyclic permutations with good reduction.
\end{theorema}

Formally, the Non-Commutative Cycle Lemma is also true for $k=0$, because then no cyclic
permutation has good reduction (as any cyclic permutation is reducible to 1). For $k\geq
1$, the statement is not so obvious because of the difference between linear and cyclic
reducibility. We will in the following always assume that $k\geq 1$.

There is a way to say which letters of $w$ remain after cyclic reduction -- as we will
see, a letter $l_i$ of $ w = l_1 \cdots l_n$ remains if the word $l_i l_{i+1} \cdots l_n
l_1 \ab\cdots l_{i-1}$ has good reduction.  Moreover, we also show that one can
canonically assign to each word of length $n$ that cyclically reduces to a word of length
$k$, a planar diagram, called a non-crossing circular half-pairing on $[n]$ with $k$
through strings (see Figure 5). Let us begin by recalling some definitions.

Let $n$ be a positive integer and $[n] = \{1, 2, 3, \dots, n\}$.  By a partition $\pi$ of
$[n]$ we mean a decomposition of $[n]$ into non-empty disjoint subsets $\pi = \{V_1,
\cdots, V_r\}$, i.e.
\[
[n] = V_1 \cup \cdots \cup V_r \mbox{\ and\ } V_i \cap V_j = \emptyset \mbox{\ for\ } i
\not = j
\]
The subsets $V_i$ are called the blocks of $\pi$, and we write $i \sim_\pi j$ if $i, j
\in [n]$ are in the same block of $\pi$. We say that $\pi$ has a crossing if we can find
$i_1 < i_2 < i_3 < i_4 \in [n]$ such that
\[
i_1 \sim_\pi i_3
\mbox{\ and\ }
i_2 \sim_\pi i_4,
\mbox{\ but\ }
i_1 \not \sim_\pi i_2
\]
We say that $\pi$ is \textit{non-crossing} if it has no
crossings. A partition is called a \textit{pairing} if all
its blocks have exactly two elements; this can only happen
when $n$ is even. See \cite{ns} or \cite{s} for a full discussion of
non-crossing partitions.

We next wish to consider a special kind of non-crossing partition called a half-pairing.

\begin{definition}
Let $\pi$ be a non-crossing partition in which no block has more than two elements and
for which we have at least one block of size 1. From $\pi$ create a new partition
$\tilde\pi$ by joining into a single block all the blocks of $\pi$ of size 1. If
$\tilde\pi$ is non-crossing we say that $\pi$ is a non-crossing \textit{half-pairing}.
The blocks of $\pi$ of size 1 are called the \textit{through strings}.

\end{definition}

Note that we require a half-pairing to have at least one through string. This corresponds
to $k\geq 1$ in our Cycle Lemma.

Let us relate this definition to our good reduction
problem. Let $w = l_1 \cdots l_n$ with $l_i \in \{u_1^{\pm
  1}, \dots, u_N^{\pm 1} \}$ be a word with $n$ letters that
cyclically reduces to a word of length $k$. We wish to
assign to $w$ a unique non-crossing half-pairing on $[n]$
with $k$ through strings.

\begin{center}
\hbox{}\hfill
\includegraphics{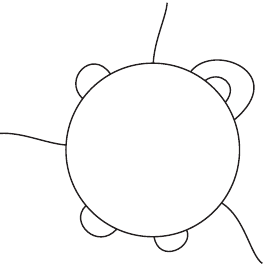} \hfill
\includegraphics{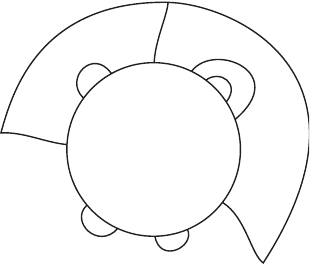}
\hfill\hbox{}
\leavevmode
\vbox{\hsize 230pt\raggedright\noindent
{\small\textbf{Figure 1.}
On the left is $\pi$ and on the right is $\tilde \pi$. }}
\end{center}

\bigskip

\begin{definition}
Let $w = l_1 \cdots l_n$ and $\pi$ be a non-crossing
half-pairing on $[n]$. We say that $\pi$ is a
$w$-\textit{pairing} if
\begin{enumerate}
\item if $(r, s)$ is a pair of $\pi$ then $l_r = l_s^{-1}$ and
\item if the singletons of $\pi$ are $(i_1), (i_2), \dots ,
  (i_k)$, then $l_{i_1} l_{i_2} \cdots l_{i_k}$ is a
  cyclic reduction of $w$.
\end{enumerate}

Given $w$ there may be more than one $\pi$ which is a
pairing of $w$. See Figure 2 below.

\begin{center}
\hbox{}\hfill
\includegraphics{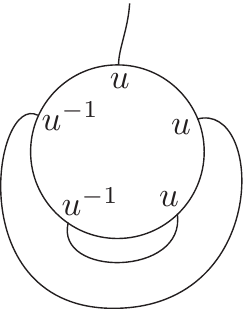}\hfill
\includegraphics{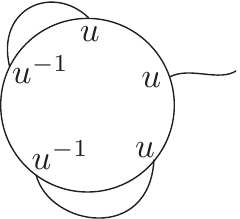}\hfill
\includegraphics{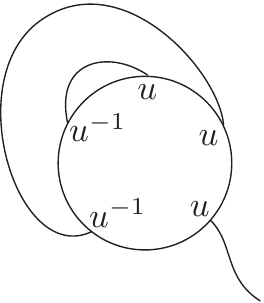}
\hfill\hbox{}

\medskip

\leavevmode
\vbox{\hsize 290pt\raggedright\noindent
{\small\textbf{Figure 2.}
The three possible $w$-pairings of $u\, u\, u\, u^{-1} u^{-1}$.}}
\end{center}

\end{definition}

In order to have a unique half-pairing associated with a
word we impose a third condition which, in particular, will
exclude the second and third diagrams in Figure 2.

\begin{definition}\label{outward}
Let $\pi$ be a non-crossing half-pairing of $[n]$. To each $i \in [n]$ we assign an
orientation: out or in. Each singleton is assigned the \textit{out} orientation. For each
pair $(r, s)$ of $\pi$ exactly one of the cyclic intervals $[r,
  s]$ or $[s, r]$ contains a singleton (recall that we have at least one
  singleton). If $[r, s]$ does
\textit{not} contain a singleton, then we assign $r$ the
\textit{out} orientation and $s$ the \textit{in}
orientation.

\end{definition}

\begin{definition}
Let $\pi$ be a non-crossing half-pairing. We say that $i$
covers the letter $j$ if both have the out orientation and
either $i+1 = j$, or $\pi$ pairs each letter of the cyclic
interval $[i+1, j-1]$ with some other letter in the cyclic
interval $[i+1, j-1]$. In particular this means that $\pi$
has no singletons in $[i+1, j-1]$.
\end{definition}

\begin{center}

\includegraphics{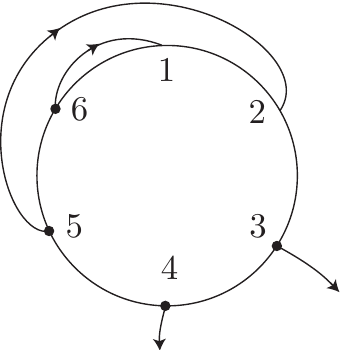}
\medskip

\leavevmode \vbox{\hsize 290pt\raggedright\noindent {\small\textbf{Figure 3.} The four
outward oriented points are 3, 4, 5, and 6. 3 covers 4, 4 covers 5 and 3, and 5 covers
6.}}
\end{center}

\begin{definition}
Let $w = l_1 \cdots l_n$ be a word and $\pi$ a non-crossing half-pairing of $[n]$. We say
that $\pi$ is $w$-\textit{admissible} if it is a $w$-pairing and we have for all $i$ and
$j$, $l_i \not = l_j^{-1}$ whenever $i$ covers $j$.

\end{definition}

\begin{center}

\includegraphics{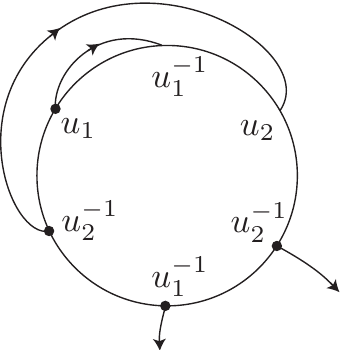}\quad
\raise2em\vbox{\hsize 220pt\raggedright\noindent
  {\small\textbf{Figure 4.} Let $w = u_1^{-1} u_2 u_2^{-1}
    u_1^{-1} u_2^{-1} u_1$ and $\pi = \{ (1, 6), (2, 5),
    (3), (4) \}$. $\pi$ is $w$-admissible. The second and
    third diagrams in Figure 2 are not $w$-admissible. }}
\end{center}

We shall show in Theorem \ref{uniqueness} that every word has a unique $w$-admissible
half-pairing. One way to obtain it is shown in Figure 5 below for the word in $u=u_1$ and
$v=u_2$ given by $w = uvv^{-1}u^{-1}v^{-1}u^{-1}$.

If a word $w$ has good reduction then the algorithm in the caption of Figure 5 produces a
periodic pattern from the start.  If a word $w$ has a cyclic permutation $w'$ which has
good reduction then use the unique $w'$-admissible non-crossing half-pairing of $w'$. It
will be a theorem that the resulting partition is independent of which cyclic rotation we
choose.

Conversely, given a $w$-admissible half-pairing with $k$
through strings, we shall see that the cyclic permutations
that start with one of these $k$ through strings will be the
permutations with good reduction.

These diagrammatic results will enable us to prove the
theorem below.

\begin{definition}
Given a word $w$ let $v$ be the letters of $w$ to which the
through strings of the unique $w$-admissible half-pairing
are attached. Then $v$ is a cyclic reduction of $w$ -- we
shall call it the \textit{standard} cyclic reduction of $w$
and denote it $\wh{w}$.
\end{definition}

\begin{theorema}
Let $k\geq 1$ and $v$ be a cyclically reduced word of length $k$. The number of words in
$\mathbb{F}_N$ of length $n$, whose standard cyclic reduction is $v$, is
$$(2N - 1)^{(n-k)/2} \ab \times
\binom{n}{(n-k)/2}.$$
In particular, this number does not depend on $v$.

\end{theorema}

\vbox{\begin{center}
\includegraphics{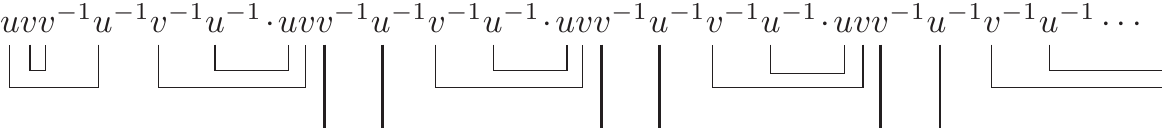}

\bigskip

\leavevmode \vbox{\hsize 350pt\raggedright\noindent
  {\small\textbf{Figure 5.}  Let $w =
    uvv^{-1}u^{-1}v^{-1}u^{-1}$ and $w_\infty$ be the word
    $w$ repeated infinitely often. In the figure above the
    repetitions of $w$ are separated by a `$\cdot$'. We
    start pairing from the left, searching for the first
    pair of adjacent letters that are inverses of each
    other. In this example it is the second and third
    letters. As soon as we find the first pair we return to
    the left and begin searching again, skipping over any
    letters already paired, in this case it is the first and
    fourth letters. Some letters will never get paired and
    these become the through strings. Eventually the pattern
    of half-pairings becomes periodic, this gives the unique
    $w$-admissible pairing. See Theorem \ref{uniqueness}. In
    this example the pattern of half-pairings becomes
    periodic after the fourth letter.}}

\end{center}}

\begin{remark} Recall that $\mathbb{F}_N$ is the free group
on the generators $u_1, u_2, \dots,\ab u_N$ and that $\mathbb{C}[\mathbb{F}_N]$ is the
group algebra of $\mathbb{F}_N$. Let $x = u_1 + u_1^{-1} + \cdots + u_N + u_N^{-1}$.  By
$\wh{x^n}$ we mean the application of the standard cyclic reduction to each word in the
expansion of $x^n$. Let $Q_k$ be the element of $\mathbb{C}[\mathbb{F}_N]$ which is the
sum of all cyclically reduced elements of length $k$. By the theorem above each word in
$Q_k$ is the standard cyclic reduction of the same number of words in the expansion of
$x^n$. Thus when we partition the set of words in $x^n$ that cyclically reduce to a word
of length $k$, into subsets according to which is their standard cyclic reduction, all
the equivalence classes have the same number of elements, namely $s_{n,k} = (2N -
1)^{(n-k)/2} \ab \times \binom{n}{(n-k)/2}$, when $n - k$ is even, and 0 when $n - k$ is
odd (for $n > 0$ and $k > 0$).  Hence we have the following corollary. Note that the
number, $s_{n,0}$, of words in $x^n$ that are reducible to 1 doesn't follow the simple
rule above; indeed, the sequence $\{ s_{n,0}\}_n$ is the moment sequence of the
distribution of $x$, which is the so-called Kesten measure, see \cite{k}.
\end{remark}

\begin{corollary}\label{xtoQ}
\[\wh{x^n} = Q_n + s_{n, n-2} Q_{n-2} + \cdots +
\bigg\{ \begin{array}{cc} s_{n,0} & $n$\ \mbox{even}
  \\ s_{n,1} Q_1 & $n$\ \mbox{odd} \end{array}
\]
\end{corollary}


\section{Proof of Main results}
\begin{notation}
Let $w$ be a word of length $n$ and $w_\infty$ the infinite word $w\, w\, w\, w \cdots$
obtained by repeating $w$ infinitely many times. Recall that a word is \textit{reducible
to 1} if it linearly (equivalently, cyclically) reduces to the identity in $\F_N$. If $w
= l_1 \cdots l_n$ is a word we let $w^{-1} = l_n^{-1} \cdots l_1^{-1}$; i.e. we reverse
the string and take the inverse of each letter but do not do any reduction. Given a word
$w$ we let $|w|$ be the length of the linear reduction of $w$.
\end{notation}

\begin{remark}\label{standardform}
Let $w$ be a word whose linear reduction is not cyclically reduced. Then either the first
and last letter of $w$ must cancel each other, or this cancellation must happen after
removing a prefix or a suffix which is reducible to 1. By repeatedly cancelling letters
at the ends of $w$ and removing prefixes or suffixes which are reducible to 1, we are
left with a word $\wt{w}$ which neither has a prefix or suffix which is reducible to 1
nor has cancellation of the first and last letters. For such a $\wt{w}$ the cyclic and
linear reduction are the same. Thus for every word $w$ there is a word $\wt w$ whose
linear reduction is cyclically reduced and words $x$ and $y$ such that $x y$ is reducible
to 1 and such that we have as a concatenation of strings $w = x \wt w y$. Depending on
the order of cancellation different decompositions of a word may be found -- we only
require the existence of such a decomposition. See Figure 6 for an example.

\end{remark}

\begin{proposition}\label{shift-translation}
Let $w$ be a word of length $n$ which reduces cyclically to a word of length $k > 0$. Let
$s$ be a prefix of $w_\infty$. If the number of letters in $s$ exceeds $n(1 + n/k)$ then
$|ws| = k + |s|$.
\end{proposition}

\begin{proof}
Write $w = x \wt{w} y$ with $x$, $y$, and $\wt{w}$ as in
Remark \ref{standardform} above. For each positive integer
$m$ the length of the linear reduction of $(\wt{w})^m$ is $m
k$.

Now the linear reduction of $x$ and the linear reduction of
$y$ are inverses of each other, so the last letter of the
linear reduction of $x$ is the inverse of the first letter
of the linear reduction of $y$. Thus, if there is any
cancellation between $x$ and $\wt{w}^m$ there can be none
between $\wt{w}$ and $y$, and vice versa for cancellation
between $\wt{w}^m$ and $y$, i.e. if there is cancellation
between $\wt{w}$ and $y$ there can be none between $x$ and
$\wt{w}$.

Let $s$ be a prefix of $w_\infty$ with $i > n(1 + n/k)$ letters. We shall show that $|w
s| = k + |s|$. Let $m = [i/n]$. Since $i > n(1 + n/k)$, we have $m > n/k$.

First, suppose there is cancellation between $x$ and $\wt{w}$ but none between $\wt{w}$
and $y$. Write $s$ as $w^m s'$ with $s'$ a prefix of $w$. Since $m > n/k$ the last letter
in the linear reduction of $x \wt{w}^m$ is the last letter of $\wt{w}$. Thus $|s| = |x
\wt{w}^m y s'| = |x \wt{w}^m| + |y s'|$ and likewise
\begin{eqnarray*}
|w s |
&=&
|x \wt{w}^{m+1} y s'| = |x \wt{w}^{m+1}| + |y s'| \\
&=&
|x \wt{w}^m| + |\wt{w}| + |y s'| = k + |x \wt{w}^m| + |y s'| \\
&=&
k + |s|
\end{eqnarray*}

\bigskip
\vbox{\begin{center}
\includegraphics{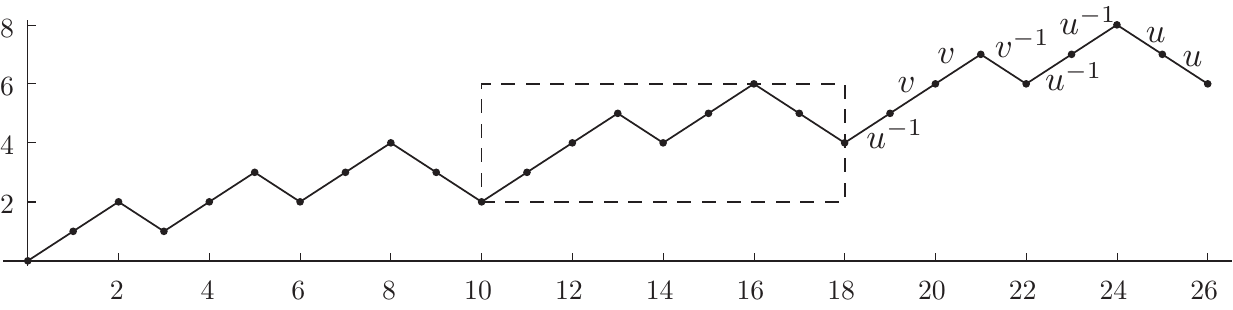}

\

\leavevmode \vbox{\hsize 320pt\raggedright\noindent
  {\small\textbf{Figure 6.} Let $w=
    uuu^{-1}vvv^{-1}u^{-1}u^{-1} = x \wt{w} y$ where $x =
    uu$, $\wt{w} = u^{-1}v$, and $y =
    vv^{-1}u^{-1}u^{-1}$. The graph of $t_i$ is shown, where
    $t_i$ is the length of the linear reduction of the first
    $i$ letters of $w_\infty$. The graph becomes
    shift-periodic at $i = 10$. The region enclosed in
    dotted lines shows one period.}}

\end{center}}

Conversely, suppose that there is cancellation between $\wt{w}$ and $y$ but none between
$x$ and $\wt{w}$. Then, as $m > n/k$, the linear reduction of $\wt{w}^m y$ begins with
the same letter as does $\wt{w}$. Hence $|s | = |x \wt{w}^m y s'| = |x| + |\wt{w}^m y
s'|$ and
\begin{eqnarray*}
|w s|
&=&
|x \wt{w}^{m+1} y s'| = |x| + |\wt{w}^{m+1} y s'| \\
&=&
|x | + |\wt{w}| + |\wt{w}^m y s'| = k + |x| + |\wt{w}^m y s'| \\
&=&
k + |s|
\end{eqnarray*}

\end{proof}


\begin{proof}[\textit{Proof of the Non-Commutative Cycle Lemma}]
Let $t_i$ be the length of the linear reduction of the first
$i$ letters of $w_\infty$. Choose $i_0 > n(1 + n/k)$. Then
by Proposition \ref{shift-translation}, for any $i \geq
i_0$, $t_{n+i} = k + t_i$. Choose $i_1$ to be the smallest
$i_1 \geq i_0$ such that $t_{i_1} < t_j$ for all $j > i_1$,
i.e. $i_1$ is the largest $i$ such that $t_{i} = t_{i_0}$.
Choose $i_2$ to be the largest $i$ such that $t_{i_2} = 1 +
t_{i_1}$. In general for $l \leq k$, choose $i_l$ to be the
largest $i$ such that $t_{i_l} = l - 1 + t_{i_1}$. Since
$t_{i_l - n} = t_{i_l} - k \leq t_{i_1}$, we must have $i_1
< i_2 < \cdots < i_k \leq i_1 + n$.

For each $l$ we choose the cyclic permutation of $w$ that starts after the ${i_l}^{th}$
letter of $w_\infty$. Since $t_{i_l + n} - t_{i_l} = k$ for such a word the linear and
cyclic reductions are the same. Also since $t$ never descends back to $t_{i_l}$, such a
word will have no prefix which is reducible to 1. Thus it has good reduction.

If we choose a cyclic permutation at an $i$ such that $t_{i} = t_{i-1} - 1$, the
resulting word will be such that its linear reduction is not cyclically reduced. Indeed,
let $s$ be the prefix of $w_\infty$ consisting of the first $i-1$ letters, and let $w_0$
be the $n$ letters following $s$. We must show that $|w_0| > k$. Since $t_i = t_{i-1} -
1$ there is cancellation between $s$ and $w_0$. Thus $t_{i-1} + |w_0| = |s| + |w_0| > |s
w_0| = t_{i + n - 1} = k + t_{i-1}$, hence $|w_0| > k$.

If we choose a cyclic permutation that starts at an $i$ for which there is $j > i$ with
$t_j = t_i$, the resulting word will also have a prefix which is reducible to 1. Thus
there are only $k$ cyclic permutations that produce good reduction.
\end{proof}

\begin{theorem}\label{uniqueness}
Let $w$ be a word of length $n$. Then there is a unique
non-crossing half-pairing on $[n]$ which is $w$-admissible.
\end{theorem}

\begin{proof}
Suppose $w = l_1 \cdots l_n$ has good reduction then we
construct the unique half-pairing which is $w$-admissible as
follows. Starting with $l_1$ and moving to the right find
the first $i < n$ such that $l_i = l_{i+1}^{-1}$. Pair these
elements and return to $l_1$ and repeat the process,
skipping over any letters already paired.  Continue passing
through $w$ until no further pairings can be made. See Figure 5. Put
through strings on any unpaired letters. This produces a
half-pairing which we denote $\pi$.  As the pairs only
involve adjacent letters or pairs that are adjacent after
removing an adjacent pair no crossings will be
produced. Moreover no pair $(r, s)$, $r < s$, will be
produced with an unpaired letter in between. Thus the
partition will be a non-crossing half-pairing.

To show that $\pi$ is $w$-admissible we must show that there
are no $i$ and $j$ such that $i$ covers $j$ and $l_i =
l_j^{-1}$. Suppose $i$ covers $j$, then according to the
definition, both have the out orientation.

Let us break this into two cases. First case: $i$ is a
through string. Since $i$ covers $j$, either $j= i+1$ or
$\pi$ pairs every point of the cyclic interval $[i+1, j-1]$
with another point of $[i+1, j-1]$. In the first of these
possibilities the algorithm would have paired $i$ with $j$
unless $i = n$, but this would imply that the linear and
cyclic reduction of $w$ are not the same. Thus we are left
with the case that $\pi$ pairs every point of the cyclic
interval with another point in $[i+1, j-1]$.

Since $w$ has good reduction, $\pi$ starts with a through string -- else $w$ would have a
prefix which is reducible to 1. Thus we must have $i < j$ since otherwise the cyclic
interval $[i+1,j-1]$ would contain a through string. Hence each number in the interval
$[i+1, j-1]$ is paired by $\pi$ with another number in the interval $[i+1, j-1]$. Now our
algorithm would have paired $l_i$ with $l_j$, so we cannot have $l_i = l_j^{-1}$.

The second case is when $i$ is the opening point of a pair
$(i, j')$ of $\pi$. We must have $i < j' < j$, for otherwise
our algorithm would have paired $i$ with $j$. However this
contradicts our assumption that each number in the interval
$[i+1, j-1]$ is paired by $\pi$ with another in the
interval. Thus $\pi$ is $w$-admissible.

To see that $\pi$ is unique notice that each time we add a
pair it is a forced move. Indeed suppose $i$ is the first
$i$, starting from the left, such that $l_i =
l_{i+1}^{-1}$. We cannot pair $l_i$ with any earlier element
because that would imply the earlier existence of an
adjacent pair; we cannot pair $l_i$ with any later letter as
this would force $i$ to cover $i+1$. We then look for the
next pair of elements either adjacent or adjacent after
skipping over $\{i, i+1\}$. By the same argument this
pairing is also forced and continuing in this way we see
that all pairs are forced and thus there is only one
$w$-admissible half-pairing.

Now suppose that $w$ does not have good reduction. By the Non-Commutative Cycle Lemma
there are $k$ cyclic permutations of $w$ which have good reduction -- one for each
through string. Indeed, each cyclic permutation of $w$ which has good reduction begins
with a through string. Between each pair of through strings the method for pairing the
elements is always the same: start at the first letter to the right of the through string
and pair the first pair of adjacent letters that are inverses of each other and then
return the the through string and repeat. Thus the method of pairing is entirely `local'
and is independent of at which through string we begin.
\end{proof}

\begin{lemma}\label{count1}
The number of non-crossing half-pairings on $[n]$ with $k$
through strings is $\binom{n}{(n-k)/2}$.
\end{lemma}

\begin{proof}
We use the method introduced in \cite{kms}. We place the
points $1, 2, 3, \dots, n$ around the outside of a circle in
clockwise order. On each point we shall place either a black
dot or a white dot with a total of $(n-k)/2$ black dots and
$(n+k)/2$ white dots. There are $\binom{n}{(n-k)/2}$ ways of
doing this so we only have to show that each assignment of
dots produces a unique non-crossing half-pairing and all
half-pairings are produced in this way.

Now $(n-k)/2$ will be the number of pairs in the half
pairing and each black dot will indicate which of the two
points of the pair has the outward orientation. Starting at
any black dot and moving clockwise search for the first
available white dot not already paired with a black dot ---
except every time we pass over a black dot we skip a white
dot to leave a white dot for the black dot to pair with. We
proceed until all black dots are paired. Any remaining white
dots become through strings.

Conversely starting with a non-crossing half-pairing on
$[n]$ with $k$ through strings, put a white dot on each
through string and a white dot on the point of each pair
with the inward orientation. Finally put a black dot on the
point of each pair with the outward orientation. This gives
the bijection between diagrams and dot patterns.
\end{proof}

\begin{center}
\hbox{}\hfill
\raise0.7em\hbox to 65pt{\hfil\includegraphics{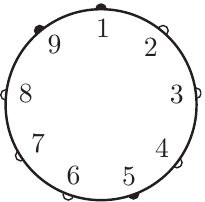}\hfil} \hfill
\includegraphics{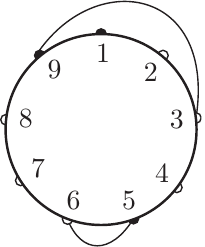}\hfill
\includegraphics{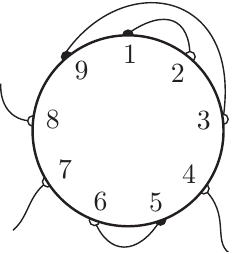} \hfill
\hbox{}

\

\leavevmode \vbox{\hsize 220pt\raggedright\noindent
{\small\textbf{Figure 7.} At the left a dot diagram, in the
centre the same diagram with a few strings added, and at the
right the completed diagram.}}
\end{center}

\begin{theorem}\label{counting-theorem}
Let $v$ be a cyclically reduced word of length $k$. The
number of words of length $n$, whose standard cyclic
reduction is $v$, is $(2N - 1)^{(n-k)/2} \ab \times
\binom{n}{(n-k)/2}$, in particular this number does not
depend on $v$.
\end{theorem}

\begin{proof}
Let $w = l_1 \cdots l_n$ be a word of length $n$ whose
standard cyclic reduction is $v = v_1 \cdots v_k$.  By
Theorem \ref{uniqueness} there exists $\pi$, a unique
non-crossing half-pairing $\pi$ on $[n]$ with $k$ through
strings which is $w$-admissible. If the through strings of
$\pi$ are at $i_1, \dots, i_k$ then $v_j = l_{i_j}$. We
shall say that the $j^{th}$ through string is
\textit{coloured} $v_j$. If $(r, s)$ is a pair of $\pi$ and
$r$ has the outward orientation then we shall say the pair
$(r, s)$ is \textit{coloured} $l_r$. Thus each word whose
standard cyclic reduction is $v$ is associated with a unique
non-crossing half-pairing coloured with the letters $\{u_1,
u_1^{-1}, \dots, u_N, u_N^{-1}\}$ subject to the rule that
no outward oriented point has the inverse colour of a point
by which it is covered.

It remains to count how many of these coloured diagrams
there are.  By Lemma \ref{count1} there are
$\binom{n}{(n-k)/2}$ diagrams with $k$ through strings. The
through strings are always coloured by the letters of $v$,
so there is no choice here. However the outward oriented
point of each pair can be coloured by any letter in $\{ u_1,
u_1^{-1}, \dots, u_N, u_N^{-1} \}$ except the inverse of the
colour that covers it. Thus there are $2N -1$ ways of
choosing this colour. The colour of the inward oriented
point of each pair is determined by the colour of the
corresponding outward oriented point of the pair.  Thus once
the diagram is selected there are $(2N - 1)^{(n-k)/2}$ ways
of colouring it.\end{proof}


\section{Concluding Remarks}

Suppose $U_1, \dots, U_N$ are independent $m \times m$ Haar distributed random unitary
matrices. Let $u_1, \dots, u_N$ be the generators of the free group $\F_N$ and
$\bC[\F_N]$ the group algebra of $\F_N$. Let $\phi: \bC[\F_N] \rightarrow \bC$ be the
tracial linear functional defined on words in $\F_N$ by $\phi(e) = 1$, $\phi(w) = 0$ for
words $w \not = e$ and then extended to all of $\bC[\F_N]$ by linearity. Suppose $Y$ is a
linear combination of words in $\{U_1, U_1^{-1},\dots,$ $U_N,U_N^{-1} \}$ and $y$ is the
corresponding linear combination of words in $\{u_1, u_1^{-1},\dots,$ $u_N,u_N^{-1} \}$.
Voiculescu \cite{vdn} showed that
$$\lim_{m\to\infty} \E[\frac 1m \Tr(Y)] = \phi(y),$$ thus
establishing the asymptotic $*$-freeness of the $U_1,\dots,U_N$. In particular this
implies the asymptotic freeness of the self-adjoint operators $X_1, \dots, X_N$, where
$X_i = U_i + U_i^{-1}$.

In recent years the fluctuation of random matrices has been
the object of much study (see \cite{cmss, kms, mn, mss,
  ms,emrs}). Let $X = X_1 + \cdots + X_N$ and for integers $p$
and $q$ consider the asymptotic fluctuation moments.
\[
\alpha_{p,q} = \lim_n \E\Bigl[ \bigl( \Tr(X^p) - \E[\Tr(X^p)]\bigr) \cdot \bigl(
  \Tr(X^q) - \E[\Tr(X^q)]\bigr) \Bigl]
\]
One way to understand these moments is via the theory of
orthogonal polynomials. In this situation it means finding a
sequence of polynomials $\{R_k\}_k$ such that for $k \not =
l$ we have
\[
\lim_n \E\Bigl[ \bigl( \Tr(R_k(X)) - \E[\Tr(R_k(X))]\bigr) \cdot \bigl(
  \Tr(R_l(X)) - \E[\Tr(R_l(X))]\bigr) \Bigr] =0
\]
Such a sequence of polynomials is said to diagonalize the
fluctuations. This has been done for a variety of random
matrix ensembles (see \cite{j} and \cite{kms} and the
references there).

Corollary \ref{xtoQ} suggests that there ought to be polynomials $\{P_n\}_n$ such that
$\wh{P_n(x)} = Q_n$. Indeed, using the Non-commutative Cycle Lemma we have shown that the
polynomials indicated by Corollary \ref{xtoQ} do diagonalize the fluctuations of the
operator $X$ above. The polynomials can also be obtained by modifying the Chebyshev
polynomials of the first kind: let $R_0(x) = 2, R_1(x) = 1$, and $R_{k+1}(x) = x R_k(x) -
(2N-1) R_{k-1}(x)$.  Moreover for $n$ odd $P_n = R_n$ and for $n$ even $P_n(x) = R_n(x) +
2$. The proofs of these results will be presented in a subsequent paper.

\section*{Acknowledgements}
The authors would like to thank the referees for a careful reading and useful suggestions.

\bibliographystyle{plain}

\end{document}